\documentclass[12pt]{amsart}
\usepackage{amscd,amsfonts,amssymb,amsmath}
\usepackage[margin=2.5 cm]{geometry}

\usepackage{amsmath}

\newtheorem{theorem}{Theorem}[section]
\newtheorem{lemma}[theorem]{Lemma}
\newtheorem{cor}[theorem]{Corollary}
\theoremstyle{definition}
\newtheorem{definition}[theorem]{Definition}

\newtheorem{prop}[theorem]{Proposition}
\theoremstyle{remark}

\numberwithin{equation}{subsection}
\theoremstyle{plain}

\def \l { \l_{\mathcal{P}}}

\numberwithin{equation}{section}

\begin{document}
\title[ On the almost palindromic width]{On the almost palindromic width of certain free constructions of groups} 
\author{Krishnendu Gongopadhyay}
\address{Department of Mathematical Sciences, Indian Institute of Science Education and Research (IISER) Mohali,
Knowledge City, Sector 81, S.A.S. Nagar, P.O. Manauli 140306, India}
\email{krishnendu@iisermohali.ac.in}
\author{Shrinit Singh}
\address{International Center for Theoretical Sciences, Survey No. 151, Shivakote,
Hesaraghatta Hobli,
Bengaluru - 560089 }
\email{shrinit.singh@icts.res.in, shrinitsingh@gmail.com}


\date{\today}


\begin{abstract}
We provide a general structural criterion implying that a group has infinite $m$-almost palindromic width. In particular, we prove that both HNN extensions and free products exhibit infinite $m$-almost palindromic width, with the unique exception of the infinite dihedral group among free products. This framework extends and strengthens the results of \cite{MS} and \cite{GK}.
\end{abstract}
\maketitle
\section{Introduction}
Let $G=\langle X \rangle$ be a group with generating set $X$.   A reduced word $w$ in the
alphabet $X^{\pm 1}$ is called a {\it palindrome} if $w$ reads the same
left-to-right and right-to-left. Correspondingly, an element $g \in G$ is called a
\emph{palindrome} if there exists a word $w$ in $X^{\pm 1}$ representing $g$ that is a palindrome. 

More generally, an element $g \in G$ is said to be an {\it $m$-almost palindrome} if a representation of $g$ in $X^{\pm 1}$ differs from a palindrome by a change of at most $m$ letters in $X^{\pm 1}$. We denote the set of all $m$-almost palindromes in
$G$ by $\mathcal{P}_m = \mathcal{P}_m(X)$. It is clear that $\mathcal{P}_0(X)$, the set of all palindromes, generates $G$. In fact, we have an increasing sequence of generating sets for $G$ $$\mathcal{P}_0(X) \subseteq \mathcal{P}_1(X) \subseteq \mathcal{P}_2(X) \subseteq \ldots .$$ For a fixed $m$, any element $g \in G$ can be expressed as a product of $m$-almost palindromes
$$
g = p_1 p_2 \ldots p_k,
$$
where each $p_i \in \mathcal{P}_m$. The minimum $k$ of such factors is called the \emph{$m$-almost palindromic length} of $g$, denoted by $l_{\mathcal{P}_m}(g)$.  The \emph{$m$-almost palindromic width} of $G$, with respect to the generating set $X$, is given by
$$
{\rm pw}_m(G,X)= \underset{g \in G}{\sup} \ l_{\mathcal{P}_m}(g).
$$
The study of palindromic width in groups has been an active area of research (see, for instance, \cite{BST,BT}, \cite{BG,BG2,BG3}, \cite{fi}). As a generalization, Bardakov \cite[Problem~19.8]{KN} posed the following question: Does there exist a pair of natural numbers $(c,m)$ such that every element in free group of two generators $\{a,b\}$ can be written as a product of at most $c$ many $m$-almost palindromes in letters $\{a^{\pm 1}, b^{\pm 1} \}$? Staiger \cite{MS} answered this question negatively showing that no such pair exists for any non-abelian free group. This result implies that $m$-almost palindromic width of a non-abelian free group is infinite.

We prove the following theorem that gives a general criterion for infiniteness of $m$- almost palindromic width.  

\begin{theorem}\label{main}
    Let $G = \langle A \rangle$. If $G$ admits an unbounded quasimorphism $f$ that is bounded above on palindromes over $A$, then $\mathrm{pw}_m(G, A)$ is infinite.
\end{theorem}

We further apply this theorem to obtain the infiniteness of $m$-almost palindromic width for certain free constructions of groups. For  $\mathrm{HNN}$ extensions, we prove the following.
\begin{theorem}\label{HNN}
    Let $G$ be a group and $H_1$ and $H_2$ be two proper isomorphic subgroup of $G$ via isomorphism $\phi: H_1 \longrightarrow H_2$. The $\mathrm{HNN}$ extension of $G$ via $\phi$
    $$G_* = \langle G, t ~\vert ~ t^{-1}ht = \phi(h), h \in H_1 \rangle$$
    has infinite $m$--almost width with respect to the generating set $G \cup \{t,t^{-1}\}$ for every $m \geq 0$.
\end{theorem}

For free products of two groups, the following holds. 

\begin{theorem}\label{freeprod}
    Let $G = G_1 * G_2$ be a free product of $G_1$ and $G_2$ such that $|G_1| \geq 3$ and $|G_2 | \geq 2$. Then ${\rm pw}_m (G, G_1 \cup G_2)$ is infinite. 
\end{theorem}

It is evident that $${\rm pw}_0(G,X) \geq {\rm pw}_1(G,X) \geq {\rm pw}_2(G,X) \geq \ldots$$
since $\mathcal{P}_0(X) \subseteq \mathcal{P}_1(X) \subseteq \mathcal{P}_2(X) \subseteq \ldots$. This leads to the following result, a quick corollary to \cite[Proposition~3.13]{GK}.

\begin{cor}\label{fincor}
    Let $G = G_1*_H G_2$ be an amalgamated free product of $G_1$ and $G_2$ over a proper subgroup $H$ such that $|G_1 : H| \leq 2$ and $|G_2 : H| \leq 2$. Then ${\rm pw}_m (G, G_1 \cup G_2)$ is finite. 
\end{cor}
As a direct consequence of the two theorems above, the $m$-almost palindromic width of the fundamental group of a certain graph of groups can be determined. Using the same notation as in \cite{GK}, we obtain the following corollary. 

\begin{cor}\label{cor10}
    Let $X$ be a non-empty connected graph, and let $\pi_1(G,X)$ denote the fundamental group of a graph of groups $(G,X)$ with the standard generating set $S$ consisting of all vertex groups. Then ${\rm pw}_m (\pi_1(G,X), S)$ is infinite if: 
    
    \begin{enumerate}
        \item \label{cor1} $X$ is a loop with a vertex $v$ and an edge $e$ such that $G_e$ is a proper subgroup of $G_v$; or,
        
       \medskip  \item \label{cor2} $X$ is a tree, and there is an oriented edge $e = [v_1, v_2]$ such that removing $e$ while keeping $v_1$ and $v_2$ yields two disjoint graphs $X_1$ and $X_2$, with $P_i \in \text{vert}(X_i)$ satisfying the following: $G_e$ is trivial and $[\pi_1(G, X_1)] \geq 3$ and $[\pi_1(G, X_2)] \geq 2$; or,
        
     \medskip    \item \label{cor3} $X$ has an oriented edge $e = [v_1, v_2]$ such that removing the edge $e$ while keeping $v_1$ and $v_2$ does not disconnect $X$ and the resulting graph $X'$ satisfies: extending $G_e \to G_{v_i}$ to $\phi_i: G_e \to \pi_1(G, X')$ for $i = 1, 2$, we have $\phi_i(G_e) = H_i$, where $H_1$ and $H_2$ are proper subgroups of $\pi_1(G, X')$.
    \end{enumerate}
\end{cor}

In case \ref{cor1}, the fundamental group is an HNN extension of $G_v$, and the result follows from Theorem~\ref{HNN}. In case \ref{cor2}, $\pi_1(G,X)$ is a free product of $\pi_1(G,X_1)$ and $\pi_1(G,X_2)$, and the result follows from Theorem~\ref{freeprod}. The final result in case \ref{cor3} follows from Theorem~\ref{HNN}, since in this situation the fundamental group is an HNN extension of $\pi_1(G,X')$.

Some groups do have finite $m$-almost palindromic width. For instance, if $G$ is a free nilpotent group of rank $n$, then $G$ has finite $m$-almost palindromic width for every $m\geq 0$ with respect to any generating set; this follows from \cite{BG}, since finite palindromic width implies finite $m$-almost palindromic width. We conjecture that the converse direction holds as well: any group with infinite palindromic width has infinite $m$-almost palindromic width for every $m \geq 1$.

We note a technical obstacle to extending these methods to general amalgamated free products. As pointed out in \cite{BurgerMozes1997, BurgerMonod1999}, certain amalgamated free products admit no unbounded quasimorphisms at all. For such groups, the methods used in this paper necessarily fail. Nevertheless, there are amalgamated free products that do admit unbounded quasimorphisms, and in these cases the construction of a quasimorphism that is bounded on palindromes would imply the infiniteness of the $m$-almost palindromic width. 

 In this paper, we extend the main result of \cite{MS} to HNN extensions and free products, thereby generalizing earlier results of \cite{GK,BT}. As a consequence, we obtain further examples of groups that give negative answers to Bardakov’s question.

\section{Proof of Theorem \ref{main}}

 Our primary focus is the study of almost palindromic width with respect to a given generating set.  We begin by defining a quasimorphism on the group $G$. 
\begin{definition}
    A function $f : G \to \mathbb{R}$ is said to be a quasimorphism if for all $g,h \in G$ and some $D \geq 0$, we have $$|f(gh)-f(g)-f(h)| \leq D.$$
\end{definition}

Let $f$ be a quasimorphism of a group $G$. Then we prove the following lemma:
 
\begin{lemma}\label{quasilemma2}
    For any $g_1,g_2, \ldots, g_n \in G$, we have $$|f(g_1g_2 \cdots g_n) - f(g_1) - f(g_2) - \cdots - f(g_n)| \leq D(n-1).$$
\end{lemma}

\begin{proof}
    We proceed by induction on $n$.
\begin{itemize}
    \item For $ n = 1 $: $ |f(g_1) - f(g_1)| = 0 \leq D(0) $.
    \item For $ n = 2 $: $ |f(g_1g_2) - f(g_1) - f(g_2)| \leq D $, by the quasimorphism property.
\end{itemize}

\textbf{Inductive step:}
Assume the result holds for some $ n \geq 2 $, i.e., for any $ h_1, \ldots, h_n \in G $,
\[
|f(h_1 \cdots h_n) - f(h_1) - \cdots - f(h_n)| \leq D(n - 1). \tag{IH}
\]

Now consider $ g_1, \ldots, g_{n+1} \in G $. Write
\[
g_1 \cdots g_{n+1} = (g_1 \cdots g_n) \cdot g_{n+1}.
\]

By the quasimorphism property:
\[
|f(g_1 \cdots g_{n+1}) - f(g_1 \cdots g_n) - f(g_{n+1})| \leq D. \tag{1}
\]

By the induction hypothesis applied to $ g_1, \ldots, g_n $:
\[
|f(g_1 \cdots g_n) - f(g_1) - \cdots - f(g_n)| \leq D(n - 1). \tag{2}
\]

Let
\[
A = f(g_1 \cdots g_{n+1}) - f(g_1) - \cdots - f(g_{n+1}).
\]

Then,
\[
A = \left[f(g_1 \cdots g_{n+1}) - f(g_1 \cdots g_n) - f(g_{n+1})\right] + \left[f(g_1 \cdots g_n) - f(g_1) - \cdots - f(g_n)\right].
\]

By the triangle inequality and using (1) and (2):
\[
|A| \leq D + D(n - 1) = Dn.
\]

Hence, the inequality holds for $ n+1 $, completing the induction.
\end{proof}

\subsection{Proof of Theorem~\ref{main}}
    Let $f$ be an unbounded quasimorphism of $G$ such that $|f(gh)-f(g)-f(h)| \leq D$ for some $D \geq 0$. Since $f$ is bounded above on palindromes of $A$, we assume that $f(p) \leq c$ for some $c \geq 0$ and for all $p \in \mathcal{P}_0(A)$.  First, we prove that $f$ is bounded on $m$-almost palindromes.

    Let $p$ be a palindrome in which changing $r$ for $r \leq m$ letters gives an $m$-almost palindrome $\Tilde{p}$.
    
    Let $p = a_1s_1a_2s_2 \ldots a_rs_ra_{r+1}$ be the palindrome where $a_i, s_j \in A $ where changing the letters $s_i \in A$ for $i \in \{1, \ldots r \}$ to $\Tilde{s_i}$ gives $\Tilde{p} = a_1\Tilde{s_1}a_2\Tilde{s_2} \ldots a_r\Tilde{s_r}a_{r+1}$.
    From Lemma \ref{quasilemma2}, we have 
    $$f(p) - \sum_{i=1}^{i=r+1} f(a_i) -  \sum_{i=1}^{i=r}f(s_i) \geq -2D.r$$

    Since $f(p) \leq c$ and $f(s_i) \leq c$, we have $$ (r +1)c - \sum_{i=1}^{i=r+1} f(a_i) \geq -2Dr.$$

    We have
    \begin{equation}
        - \sum_{i=1}^{i=r+1} f(a_i) \geq -2Dr - (r+1)c
    \end{equation}

    $$f(\Tilde{p}) = f(a_1\Tilde{s_1}a_2\Tilde{s_2} \ldots a_r\Tilde{s_r}a_{r+1})$$
    $$\leq \sum_{i=1}^{i=r+1} f(a_i) + \sum_{i=1}^{i=r}f(\Tilde{s_i}) + 2Dr$$
    $$= 2Dr+(r+1)c+rc+2Dr = 4Dr+ (2r+1)c \leq 4Dm+ (2m+1)c.$$
    We have $f(\Tilde{p}) \leq 4Dm+ (2m+1)c.$ 

    We have proved that $f$ is bounded above by $$4Dm+ (2m+1)c$$
    on $m$-almost palindromes.  Suppose, on the contrary, if $\mathrm{pw}_m(G,A) < \infty$. Then every element of $G$ would have finite $f$ value. This follows from Lemma \ref{quasilemma2}. Since $f$ is unbounded on $G$, this is not possible. Hence $\mathrm{pw}_m(G,A)$ is infinite.    
\qed

\section{Proof of Theorem~\ref{HNN}} 
\textbf{HNN extension:} Let $G$ be a group and let $H_1$ and $H_2$ be two proper isomorphic subgroup of $G$ with the isomorphism $\phi: H_1 \longrightarrow H_2$. The $\mathrm{HNN}$ extension of $G$ via $\phi$ is given by
    $$G_* = \langle G, t ~\vert~ t^{-1}ht = \phi(h), h \in H_1 \rangle .$$

Any element $w \in G_*$ can be represented in the form $$w = g_0t^{\beta_1}g_1t^{\beta_2}\ldots g_{r-1}t^{\beta_{r}}g_r,$$ 
where $g_i \in G$ and $\beta_i \in \{+1,-1\}$ for $i \in \{1, \ldots , r \}$. This representation of $w$ is said to be reduced if no subword $t^{-1}g_it$ where $g_i \in H_1$ or $tg_it^{-1}$ where $g_i \in H_2$ exists within the representation.  Such representation of an element of an HNN extension is not unique, but Britton's lemma ensures that if $w$ is reduced and $r \geq 1$, then $w$ is non-trivial in $G_*$. The following lemma highlights a uniqueness among different reduced representations of same element,  which is essential for the proof of Theorem \ref{HNN}.

 \begin{lemma}\cite[Lemma~3]{bardakov1997width}\label{uniHNN}
     Let $w_1 = g_0t^{\beta_1}g_1t^{\beta_2}\ldots g_{r-1}t^{\beta_{r}}g_r$ and $w_2 = h_0t^{\eta_1}h_1t^{\eta_2}\ldots h_{n-1}t^{\eta_{n}}h_n$ be two reduced elements in $G_*$ such that $w_1 = w_2$. Then $r=n$ and $\beta_i = \eta_i$ for all possible $i$.
 \end{lemma}

\begin{definition}
    The {\it signature} of an element $w = g_0t^{\beta_1}g_1t^{\beta_2}\ldots g_{r-1}t^{\beta_{r}}g_r \in G_*$ (in reduced form) is the sequence $$sgn(w) =  (\beta_1,\ldots ,\beta_r).$$ 
\end{definition}

 By Lemma \ref{uniHNN}, signature of any element $w \in G_*$ is unique. Let $w \in G_*$ such that $sgn(w) = (\beta_1,\ldots ,\beta_r)$, then the signature of its inverse is $$sgn(w^{-1}) = (-\beta_r,\ldots , -\beta_1) .$$ 
 
 Let $\sigma = (\theta_1,\ldots , \theta_n)$ be a signature, then we define the length of the signature, $|\sigma| = n$ and the inverse of the signature $\sigma^{-1} = (-\theta_n,\ldots ,-\theta_1).$ So $sgn(w^{-1}) = (sgn(w))^{-1}.$ 

Let $\sigma$ and $\tau$ be two signatures. We define a product $\sigma\tau$ of signatures $\sigma$ and $\tau$ to be a sequence by appending $\tau$ after $\sigma$. Suppose $\sigma = \sigma_1\rho$ and $\tau = \rho^{-1}\tau_1$, we define an $s$-product, $\sigma[s]\tau = \sigma_1\tau_1$. The following lemma holds.

\begin{lemma}\cite[Lemma~4]{bardakov1997width}
Let $w_1,w_2 \in G_*.$ Then there exists an integer $s \geq 0$ such that $sgn(w_1w_2) = sgn(w_1)[s]sgn(w_2),$ with $sgn(w_1) = \sigma_1\rho$ and $sgn(w_2) = \rho^{-1}\tau_1$ and $|\rho| = s.$
\end{lemma}

Let $sgn(w) = (\theta_1,\ldots , \theta_r)$ be the signature of $w \in G_*.$ We define the following:
\begin{itemize}
    \item $p_k(w) :=$ the number of \emph{maximal} consecutive 
    $+1,+1,\ldots,+1$ blocks of length $k$ in the signature of $w$;
    
    \item $m_k(w) :=$ the number of \emph{maximal} consecutive 
    $-1,-1,\ldots,-1$ blocks of length $k$ in the signature of $w$;
    
    \item $d_k(w) := p_k(w) - m_k(w);$
    \item $r_k(w) := $ remainder of $d_k(w)$ divided by $2,$ and, let's define a function $f$ on $G_*$.
\end{itemize}
Now, we define the function $f$ on $G_*$ by
$$f(w) := \Sigma_{k=1}^{\infty} r_k(w).$$

It is easy to see that $p_k(w^{-1}) = m_k(w)$, and hence $d_k(w) + d_k(w^{-1}) = 0$ for all $w \in G_*.$

 \begin{lemma}\cite[Lemma~9]{bardakov1997width}\label{hnnquasilemma}
 $f$ is a quasi-homomorphism; in fact, for any $w_1,w_2 \in G_*$, $$|f(w_1w_2) - f(w_1) - f(w_2)| \leq 6.$$
 \end{lemma}

\begin{definition}
    Let $w = g_0t^{\beta_1}g_1t^{\beta_2}\ldots g_{r-1}t^{\beta_{r}}g_r$ be a reduced element in $G_*$. Define $$\Bar{w} = g_rt^{\beta_r}g_{r-1}t^{\beta_{m-1}}\ldots g_1t^{\beta_{1}}g_0.$$
    We say $w$ is a group-palindrome if $\Bar{w} = w$ as a word.
\end{definition}

 \begin{lemma}\cite[Lemma~2.7]{GK}\label{hnnpali}
     Let $p \in \mathcal{P}_0(G \cup \{t, t^{-1}\})$ be a palindrome in $G_*$. Then $f(p) \leq 1. $
 \end{lemma}

 \begin{lemma}\label{hnnunbound}
     The quasimorphism $f$ is an unbounded on $G_*$.
 \end{lemma}

 \begin{proof}
     The proof directly follows from the proof of \cite[Theorem~1.1]{GK}. In fact, Gongopadhyay and Krishna explicitly gave elements of $G_*$ with arbitrary large value of $f$. Hence $f$ is unbounded on $G_*$.
 \end{proof}

 \subsection{Proof of Theorem \ref{HNN}}
     Now the proof is direct corollary of Theorem \ref{main} using Lemma \ref{hnnpali} and Lemma \ref{hnnunbound}.
\qed

\section{Proof of Theorem \ref{freeprod}}

In this section, our aim is to prove Theorem~\Ref{freeprod}. We begin by recalling several definitions and notations that will be used throughout.

 Let $G = G_1 * G_2$ be the free product of two groups $G_1$ and $G_2$.  
Set $A = G_1 \setminus \{1\}$ and $B = G_2 \setminus \{1\}$.  
Every nontrivial element $g \in G$ admits a unique reduced expression
\[
g = x_1 y_1 x_2 y_2 \cdots x_r y_r \quad \text{or} \quad 
g = x_1 y_1 x_2 y_2 \cdots x_r,
\]
where each $x_i, y_i \in A \cup B$ and no two consecutive letters lie in the same free factor.  
We now introduce the terminology used in the construction of code quasimorphisms for free products; see \cite{KB} for more details.

\begin{definition}
Let $g \in G$ be written in reduced form.  
The \emph{$A$-tuple} of $g$ is the subsequence of letters lying in $A$.  
The \emph{$A$-code} is the sequence of positive integers obtained by replacing each maximal block of identical letters in the $A$-tuple by its length.  
The \emph{$B$-tuple} and \emph{$B$-code} are defined analogously.
\end{definition}

For a tuple $z = (n_1,\dots,n_k)$ of positive integers and $C \in \{A,B\}$, define  
\[
\theta_z^C(g)
\]
to be the maximum number of disjoint occurrences of $z$ as a consecutive subtuple of the $C$-code of $g$.  
Writing $\bar z = (n_k,\dots,n_1)$, the associated \emph{code quasimorphism} is
\[
f_z^C(g) := \theta_z^C(g) - \theta_{\bar z}^C(g).
\]

If $G_1 \cong \mathbb{Z}$, we replace the $A$-tuple by its \emph{weighted $\mathbb{Z}$-code}, obtained by grouping consecutive integers of the same sign and replacing each block by the absolute value of its sum.  
This yields $\theta_z^{\mathbb{Z}}(g)$ and the quasimorphism
\[
f_z^{\mathbb{Z}}(g) := \theta_z^{\mathbb{Z}}(g) - \theta_{\bar z}^{\mathbb{Z}}(g).
\]

\begin{definition}
    
If $\psi\colon G \to \mathbb{R}$ is a quasimorphism, its
\emph{homogenisation} $\bar{\psi} : G \to \mathbb{R}$ is defined by
\[
\bar{\psi}(g) := \lim_{n \to \infty} \frac{\psi(g^n)}{n}.
\]
\end{definition}

\begin{lemma}\cite[Lemma~2.4]{KB}\label{homo}
    The homogenisation of a quasimorphism $\psi$ is again a  quasimorphism and it satisfies $|\bar{\psi}(g)-\psi(g)| \leq D(\psi)$, where $D(\psi)$ is the defect of $\psi$.
\end{lemma}

The proof of Theorem~\ref{freeprod} relies on the code quasimorphisms introduced by Karlhofer.

\begin{lemma}\cite[Lemma 4.8]{KB}
If neither $ G_1 $ nor $ G_2 $ is infinite cyclic, then $ f_z^C $ is a quasimorphism with defect at most 30.
\end{lemma}

\begin{lemma}\cite[Lemma 5.5]{KB} 
If $ G_1 = \mathbb{Z} $, then $ f_z^{\mathbb{Z}} $ is a quasimorphism with defect at most 30.
\end{lemma}

A tuple $z = (n_1, \cdots , n_k)$ of non-zero natural numbers  is called \emph{generic} if $\bar{z}$ does not appear as a tuple of $k$ adjacent numbers in $z^2 = (n_1,\cdots , n_k, n_1, \cdots, n_k)$.  

\begin{prop}\cite[Proposition 4.11]{KB}\label{noncyclic}
If neither $ G_1 $ nor $ G_1 $ is infinite cyclic and $ z $ is generic, then the homogenized quasimorphism $ \bar{f}_z^C $ (or $ \bar{f}_z^{G_1} + \bar{f}_z^{G_2} $ if $ G_1 \cong G_2 (\not\cong \mathbb{Z}/2)$) is an unbounded quasimorphism.
\end{prop}

\begin{prop}\cite[Proposition 5.7]{KB}\label{cyclic}
If $ G_1 = \mathbb{Z} $ and $ z $ is generic, then the homogenized quasimorphism $ \bar{f}_z^{\mathbb{Z}} $ is an unbounded quasimorphism.
\end{prop}

\begin{lemma}\label{freeprodpali}
   The quasimorphisms mentioned in Propositions \ref{noncyclic} and \ref{cyclic} are bounded on the set of all palindromes formed from elements of $G_1$ and $G_2$.
\end{lemma}

\begin{proof}
  We argue case by case.

  \smallskip
\textbf{Case 1: Neither $G_1$ nor $G_2$ is infinite cyclic.}  
Let $z$ be a generic tuple.  
By Proposition~\ref{noncyclic}, there exists an unbounded homogeneous quasimorphism
$\psi$, equal to $\bar{f}_z^{G_1}$ if $G_1 \not\cong G_2$, or  
$\bar{f}_z^{G_1} + \bar{f}_z^{G_2}$ if $G_1 \cong G_2$.

If $g$ is a palindrome, then its reduced word is palindromic \cite{GK}.  
That means the associated $A$-codes and $B$-codes are respectively identical with thier reverses, i.e.,
\[
\theta_z^C(g) = \theta_{\bar z}^C(g) \qquad \text{for } C \in \{A,B\}.
\]
Hence  $f_z^C(g)=0$ and therefore $\psi(g)=0$ for a palindrome $g$.  
By Lemma~\ref{homo}, we have
\[
|\bar{\psi}(g)| \le D(\psi),
\]
so $\bar{\psi}$ is bounded on palindromes while unbounded on $G$.

\smallskip
\textbf{Case 2: One factor is infinite cyclic.}  
Assume $G_1 \cong \mathbb{Z}$.  
Let $z$ be generic.  
By Proposition~\ref{cyclic}, $\bar{f}_z^{\mathbb{Z}}$ is an unbounded quasimorphism.

For a palindrome $g$, the weighted $\mathbb{Z}$-code is symmetric, so
\[
\theta_z^{\mathbb{Z}}(g)=\theta_{\bar z}^{\mathbb{Z}}(g)
\quad\Rightarrow\quad
f_z^{\mathbb{Z}}(g)=0.
\]
And Lemma~\ref{homo} shows that $\bar{f}_z^{\mathbb{Z}}$ is bounded on palindromes.
\end{proof}

\subsection{Proof of Theorem~\ref{freeprod}} The proof follows from Proposition \ref{noncyclic}, Proposition  \ref{cyclic}, Lemma \ref{freeprodpali}, and Theorem \ref{main}.
\qed

\subsection{Almost Palindromic Width for Certain Amalgamated free product}

We now consider amalgamated free product. Let $G_1$ and $G_2$ be  groups with subgroups $H_1 \leq G_1$ and $H_2 \leq G_2$ and let $\phi : H_1 \longrightarrow H_2$ be an isomorphism. The amalgamated free product of $G_1$ and $G_2$ identifying $H_1$ and $H_2$ via the isomorphism $\phi$ is the group $$G= \langle G_1, G_2 ~|~ \phi(h) = h, h \in H_1 \rangle.$$ 
Since $H_1$ and $H_2$ are identified, we denote them by $H$ and $G$ by $G_1 *_H G_2$.

The next lemma is used to pass almost-palindromic width to quotients. We omit the proof as it is straightforward.
\begin{lemma}\label{epi}
    Let $G = \langle A \rangle$ and $K = \langle B \rangle$ be two groups, where $\mathcal{P}_m(A)$ is the set of $m$-almost palindromes in the alphabet $A^{\pm 1}$, and $\mathcal{P}_m(B)$ be the set of $m$-almost palindromes in the alphabet $B^{\pm 1}$. If $\phi : G \longrightarrow K$ is an epimorphism such that $\phi(A) = B$ then $\mathrm{pw}_m(K,B) \leq  \mathrm{pw}_m(G,A)$.
\end{lemma}
 The following corollary follows from \cite[Proposition~3.13]{GK}.

 \begin{cor}
     Let $G = G_1*_H G_2$ be an amalgamated free product such that $|G_1/H| \leq 2$ and $|G_2/H_2 | \leq 2$. Then $\mathrm{pw}_m(G, G_1 \cup G_2)$ is finite.
 \end{cor}
We also have a complementary infinite-width result. 
\begin{theorem}
      Let $G = G_1 *_H G_2$ be as above. Suppose $H$ is normal in both $G_1$ and $G_2$ and that $|G_1/H_1| \geq 3$ and $|G_2/H_2 | \geq 2$. Then $\mathrm{pw}_m(G, G_1 \cup G_2)$ is infinite.
\end{theorem}

\begin{proof}
    There is a natural surjection $G \to {G_1}/{H}*{G_2}/{H}$. By Lemma~\ref{epi}, the almost-palindromic width of $G$ with respect to $G_1 \cup G_2$ is bounded below by the almost-palindromic width of the free product ${G_1}/{H}*{G_2}/{H}$ with respect to ${G_1}/{H} \cup {G_2}/{H}$. By Theorem~\ref{freeprod},  ${G_1}/{H}*{G_2}/{H}$ has infinite almost-palindromic width, hence so does $G$.
\end{proof}

\subsection*{Acknowledgments}
The authors thank the referee for carefully reviewing the manuscript and suggesting valuable structural changes that significantly improved its clarity and readability. 

We also sincerely thank Francesco Fournier-Facio for highlighting the limitations in the application of Theorem~\ref{main}.

\end{document}